%% file: dc.tex
\documentclass[11pt,a4paper]{article}
\usepackage{color}
\usepackage[latin1]{inputenc}
\usepackage{amstext}
\usepackage{graphicx}
\usepackage{amsfonts}
\usepackage{latexsym}
\usepackage{amssymb}
\usepackage{amsthm}
\usepackage{amsmath}
\usepackage{url}

\title{Double Clustering and Graph Navigability} 
\author{Oskar
  Sandberg \thanks{The Department of Mathematical Sciences, Chalmers
    University of Technology and Göteborg University. ossa@math.chalmers.se}}

\input{defs.tex}

\begin{document}
\maketitle

\begin{abstract}
  Graphs are called \emph{navigable} if one can find short paths
  through them using only local knowledge. It has been shown that for
  a graph to be navigable, its construction needs to meet strict
  criteria. Since such graphs nevertheless seem to appear in nature,
  it is of interest to understand why these criteria should be
  fulfilled.

  In this paper we present a simple method for constructing graphs
  based on a model where nodes vertices are ``similar'' in two
  different ways, and tend to connect to those most similar to them -
  or cluster - with respect to both. We prove that this leads to
  navigable networks for several cases, and hypothesize that it also
  holds in great generality.  Enough generality, perhaps, to explain
  the occurrence of navigable networks in nature.
\end{abstract}

\input{dc_main.tex}

\section*{Acknowledgements}

Thanks to my advisors Olle Häggström and Devdatt Dubhashi for helpful
comments and corrections. Thanks also to Elchanan Mossel for comments
that helped me complete the final proof.

\bibliographystyle{plain}
\bibliography{../../tex/ossa}

\end{document}

%% file: defs.tex
\newtheorem{theorem}{Theorem}[section]
\newtheorem{lemma}[theorem]{Lemma}
\newtheorem{definition}[theorem]{Definition}

\newtheorem{corollary}[theorem]{Corollary}
\newtheorem{conjecture}[theorem]{Conjecture}

\newcommand{\PR}{\mathbf{P}}
\newcommand{\given}{\,\mathbf{|}\,}

\newcommand{\E}{\mathbf{E}}

\newcommand{\Z}{\mathbb{Z}}

\newcommand{\AND}{\text{ and }}

\newenvironment{proofof}[1]{\emph{Proof of #1:}}{\begin{flushright} $\Box$ \end{flushright}}

%% file: dc_main.tex
\section{Introduction}

Motivated by the small-world experiments of Stanley Milgram
\cite{milgram:smallworld}, and the models for social networks inspired
by them \cite{watts:smallworld}, Jon Kleinberg introduced the question
of whether graphs can be searched (or navigated) in a decentralized
manner \cite{kleinberg:smallworld}. In particular, he showed that when
a grid structure is augmented by random edges, whether it is possible
to use those edges to efficiently route queries from one point in the
grid to another depends on their distribution. In particular, if each
vertex $x$ in a $d$-dimensional grid is given one additional ``long
range'' link to some vertex, beyond those to its nearest neighbors,
then when the probability of $y$ being the selected is proportional to
$|x-y|^{-d}$, any greedy walk on the resulting graph is expected to
complete in a number of steps polylogarithmic to the graph size. If
the probability of $y$ being selected is any other exponent of the
distance (in particular $0$, meaning the long-range link is uniformly
selected) then any form of routing which uses only information about
the points seen thus far will require a number of steps which is a
fractional power of the graph size.

A natural question following from Kleinberg's results is to ask if
there is any dynamic which might cause the frequency of edges in
naturally occurring graphs to have the sought relationship with their
length. Several empirical studies of social network data following
Kleinberg have observed just such a relationship \cite{adamic:social}
\cite{nowell:geographic}, making it plausible that such a dynamic may
exist, but it has not been identified.

In this paper we observe that the desired edge distribution arises
naturally in another probabilistic model, that of best-yet sampling
from a population, and use this to show how navigable networks may
arise when vertices belong to two independent spaces and tend to
cluster in both (in social network terms, these may be identified with
the physical world and metaphorical space of ``interests'' - people
tend to be befriend those who are close in either sense.) The
resulting spatial random graph, which we dub the double clustering
graph, turns out to be navigable with respect to both spaces.

\subsection{Previous Work}

The original questions about navigability in a geographical setting
were posed and answered by Kleinberg in \cite{kleinberg:smallworld}
and \cite{kleinberg:navigation}. Later, Kleinberg
\cite{kleinberg:dynamics} and Watts et al.\ \cite{watts:social}
independently proposed similar models based on the categorization of
ideas or characteristics. The latter paper includes the idea that
vertices may be similar in several independent spaces, but does not
discuss how this might lead to the desired edge distribution.
Fraigniaud \cite{fraigniaud:greedy} went further and discussed
augmentation in more general settings based on tree-decompositions of
general graphs.

Some conceptually different work has been done previously to try to
explain the emergence of Kleinberg type edge frequencies. In
particular, \cite{clauset:navig} \cite{sandberg:evolving} and
\cite{sandberg:neighbor} propose graph rewiring processes which seem
to create navigable networks in their stationary state. These may help
explain how such networks arise under some circumstances, but are not
always an easy fit with observed reality, and have so far eluded
complete analysis.  \cite{duchon:emergence} shows a form of navigable
augmentation that depends on little knowledge of the base graph, but
this algorithm is complicated and does not give an intuitive reason
why the desired edge distribution should arise.

\subsection{Contribution}

We characterize our contribution as follows:

\begin{itemize}
\item We introduce the ``double clustering'' graph construction. This is
  a simple rule for constructing a graph between a set of vertices
  with positions in two different spaces, so that they tend to connect
  to those nearest in both. Double clustering can be seen as a spatial
  or combinatorial construction depending on whether the points are
  originally placed in graphs or metric spaces.

\item We show analytically that in several cases double clustering
  leads to navigable graphs.

\item We hypothesize that this holds for a much larger class of such
  graphs, something we illustrate with simulation of several relevant
  sub-models.
\end{itemize}

\section{Navigable Graphs}

\subsection{Decentralized Routing}

Let $G = (V,E)$ be a connected finite graph of high (some power of
$|V|$) diameter, and let the random graph $G'$ be created by addition
(augmentation) of random edges to $G$. It is well known, see for
instance \cite{bollobas:diameter}, that the diameter shrinks quickly
to a logarithm of $|V|$ when random edges are added between the
vertices.  Navigability concerns not a small diameter, however, but
rather a stronger property: the possibility of finding a short path
between two vertices in $G'$ using only local knowledge at vertex
visited. By local knowledge, one means that each vertex knows $G$, but
does not know which random edges have been added to any vertex until
it is visited.  The exact limits of such \emph{decentralized routing
  algorithms} have been discussed elsewhere
\cite{kleinberg:smallworld} \cite{barriere:efficient}, but since we
are interested only in upper bounds, we will define only the subset of
such algorithms of interest to us. When routing from for some vertex
$z$, in each step we will select as the next vertex in the path a
$G'$-neighbor of the current vertex, $x$. This choice will be made
entirely as a function of each neighbors $G$-distance to $z$, and
nothing else. All such algorithms are decentralized by Kleinberg's
definition.

The most direct decentralized algorithm, and the most important, is
\emph{greedy routing}. In greedy routing, the next vertex chosen is
that neighbor which is closest to $z$ in $G$ (with some tie-breaking
rule applied).  Note that both the original and augmented edges can be
used, but because the choice is only optimal with respect to $G$, the
path discovered by greedy routing will seldom be a minimal path in
$G'$. In one case below we will modify the routing to divert from a
greedy choice slightly for technical reasons, but the principle is
still the same.

We start with $G$ as a $d$-dimensional $n$-grid (that is $V =
\{1,2,\hdots,n\}^d$ and there are edges between adjacent vertices) and
independently add a single directed edge from each vertex to a random
destination. The long-range connection is added such that for $x, y
\in V$, and some $\alpha \geq 0$
\begin{equation}
  \PR(x \leadsto y) = {1 \over {h_{\alpha,n} |x-y|^\alpha}}
  \label{eq:gridaug}  
\end{equation}
where $x \leadsto y$ is the event that $x$ is augmented with an edge
to $y$ and $|x-y|$ denotes $L^1$ distance in $\Z^d$. $h_{\alpha,n}$ is a
normalizing constant.

The by now well known result of Kleinberg is that when $\alpha = d$,
greedy routing between any two points in $V$ takes $O(\log^2 n)$ steps
in expectation, while for any other value of $\alpha$ decentralized
algorithm creates routes of expected length at least $\Omega(n^s)$
steps for some $s > 0$ (where $s$ depends on the dimension but not the
algorithm chosen).

\subsection{Doubling Dimension and More General Augmentation}

It should be noted that if the graph is a $d$-dimensional grid as
above, and for $x \in V$ $B_r(x) = \{y \in V : |y - x| \leq r\}$,
then $|B_r(x)| \propto r^d$. For $\alpha = d$ (\ref{eq:gridaug})
can then be interpreted as
\begin{equation}
\PR(x \leadsto y) \propto {1 \over |B_r(x)|}.
  \label{eq:genaug}
\end{equation}
This general principle, that under navigable augmentation the
probability that $x$ links to $y$ should be inversely proportional to
the number of vertices that are closer to $x$ than $y$ has been
observed to hold across a wider class of graphs then just the grids,
see e.g. \cite{kleinberg:dynamics} \cite{duchon:anygraph}
\cite{nowell:geographic}, and seems to be the general principle behind
navigability. It leads directly to our first construction.

A natural generalization to grids is to study graphs which are
naturally grid-like. Let $B_r(v)$ be as above, but using graph instead
of grid distance.
\begin{definition}
  A family of graphs has \emph{bounded doubling dimension} if there
  exists a family wide constant $c$ such that for all $G = (V,E)$ in
  the family, $u, v \in V$, and $r \geq 1$
\[
B_r(u) \subset B_{2r}(u) \Rightarrow |B_{2r}(v)| \leq c |B_r(u)|.
\]
\label{def:bdd}
\end{definition}
The commonly used \emph{doubling dimension} of the family roughly
corresponds to the $\log_2$ of the smallest such $c$. This is not the
widest class of graphs where navigable augmentation is possible,
\cite{duchon:anygraph} and \cite{fraigniaud:doubling} have shown that
families with a sufficiently slowly growing dimension can still be
made navigable, but it provides a a good compromise between generality
and convenience for us.

In the constructions below, we will augment the base graph with more
than one long-range edge per vertex. In general, a $k$ edge
augmentation is expected to give $O(\log^2 n / k)$ expected greedy
routing time. Our constructions will generate close to $\log n$ edges
per vertex in expectation\footnotemark{}, and thus have routing time
$O(\log n)$.  They remind most of previously explored finite
long-range percolation models, for which the diameter is known to be
$O(\log n / \log \log n)$ \cite{coppersmith:diameter}.

\footnotetext{A degree going to infinity may seem unrealistic in terms
  of social networks, but note that $\log(6\text{ billion}) \approx
  22.5$ which is probably considerably less than the average number of
  acquaintances a person has in the real world for most definitions of
  the word ``acquaintance''. Our models can be given bounded degree by
  simply thinning the edges (removing each edge independently with
  probability $1 - 1/\log(n)$).}

\section{The Independent Interest Model}

We start by introducing a conceptually simple model. Compared to our
main model below, it is not a particularly interesting model of
networks dynamics, and not particularly realistic, but serves to
illustrate the reasoning we will use later.

Let $X_1, X_2, \hdots, X_n$ be $n$ random variables drawn from an
exchangeable joint distribution such that $\PR(X_i = X_j) = 0$ for $i
\neq j$. It is well known in this case that the probability that for
any $k$, $\PR(X_k \geq X_j\text{ for all }j < k) = 1/k$, and that this
event is independent for each $k$. This fact, combined with
(\ref{eq:genaug}) motivates the following graph model

\begin{definition} \emph{(The Independent Interest Graph)} Let $G =
  (V,E)$ be a graph, and for $x,y \in V$ let $d(x,y)$ be the graph
  (geodesic) distance between them. 

  Create the long range links as follows: For each vertex, independently
  create an exchangeable sequence of random variables $(X^x_y)_{y \in
    V}$. The add an edge from $x$ to $y$ if:
\[
X^x_y \geq X^x_z \text{ for all } z \neq x : d(x,z) < d(x, y).
\]
\end{definition}

In the social network metaphor, each $X^x_y$ in the construction above
can be seen as $x$'s interest in $y$, and the construction simply
means that $x$ befriends each $y$ who is more interesting to him than
any closer person. In other words, starting from his own position, $x$
searches outwards for friends, befriending each new person he meets if
that person is more interesting to him then the people he already
knows. In reality, of course, it is unlikely that the interest levels
$X^x_y$ would be independent for each $x$ and $y$ -- in particular,
one would expect a high correlation between $X^x_y$ and $X^y_x$. This
fact will inspire our later models below.

That the independent interest graph is navigable in fact follows from
the observations above and previous results, but for illustration we
will give a direct proof here.

\begin{theorem}
  For any family of connected graphs with bounded doubling dimension,
  the expected greedy path between any two vertices has expected
  length $O(\log n)$, where $n$ is the size of the graph.
\label{th:iimod}
\end{theorem}

\begin{proof}
  Let $z \in V$ be the target vertex. We follow the standard method:
  divide $V$ into $O(\log n)$ phases, with the $i$-th phase defined as
  the set of vertices $x$ such that $2^{i-1} < |x-z| \leq 2^i$. At a
  vertex $x$ in the $i$-th phase, for $i \geq \log \log n$, let $A$ be
  the event that $x$ has a shortcut to a lower phase, that is
\[
A = \{x \leadsto y : y \in B_{2^{i - 1}}(z)\} = \{x \leadsto B_{2^{i -
    1}}(z)\}.
\]
Let 
\[
w = \underset{v \in B_{(3/2) 2^i}(z)}{\text{argmax}}(X_v^x)
\]
By construction, $x$ has a link to $w$, so $A$ will occur if $w \in
B_{2^{i-1}}(z) \subset B_{(3/2) 2^i}(x)$. That the family has bounded
doubling dimension thus means there is a constant $c$ such that
$B_{(3/2) 2^i}(x) / B_{2^{i-1}}(z) \leq c^2$. Since each vertex in the
larger ball is equally likely to be the most interesting.
\[
\PR(A) \geq {1 \over c^2}
\]
independent of $n$ and $i$. If $A$ does not occur, then in the next
step we are by necessity not further from $z$ (nearest neighbors in
base graph are always connected), and because the edges are chosen
independently, $A$ occurs at the new vertex with at least the same
probability. Therefore the expected number of steps until $A$ occurs,
an upper bound on the number steps in a phase, is at most $c^2$.

For each sufficiently big phase, we thus have a constant bound on the
expected number of steps. Since the destination of the edges at each
vertex are independent of the previous path taken by the query, it
follows that the expected number of steps in such phases is at most
the sum over all of them, which is $O(\log n)$. Only $O(\log n)$
points in smaller phases remain, so the result holds.
\end{proof}

\section{The Double Clustering Model}

Our main model of interest is conceptually similar to the independent
interest model of the last section, but rather than letting each
vertex' interest in each other vertex be an independent random
variable, we let each vertex also live in a second space, and let the
interest between two vertices be their proximity in that space. In a
social network, this would can be represented by each individual not
only living somewhere in the physical world, but also having some
position in a less clearly defined ``space of interests'' (his job,
activities, etc.). People who live close to one another tend to become
acquainted by ``default'', while people befriend those far away only
if they interests that agree to some extent.

In the constructed graph, each vertex is thus connected to every other
vertex that is at least as ``interesting'' as any other that is at
most as ``far away''. Formally, let a distance function be a real
valued kernel $d(x,y)$ such that $d(x,x) = 0$ and $d(x,y) + d(y, z)
\geq d(x, z)$ but which is not necessarily symmetric. The most general
definition of such a graph is then:
\begin{definition}\emph{(The Double Clustering Graph)}
  Let $\{x_i\}_{i = 1}^n$ and $\{y_i\}_{i = 1}^n$ be set two sets of
  points in possibly different spaces $M_1$ and $M_2$ with distance
  functions $d_1$ and $d_2$ respectively. The graph $G = (V,E)$ is
  constructed as follows:
  \begin{itemize}
  \item $V = \{1,2,\hdots,n\}$.
\item $(i,j) \in E$ if for all $k \in V$, $k \neq i,j$:
\[
 d_1(x_i ,x_k) < d_1(x_i, x_j) \Rightarrow
  d_2(y_i,y_k) \geq d_2(y_i, y_j)
\]
  \end{itemize}
\label{def:gdc}
\end{definition}

\begin{figure}
  \centering
  \resizebox{4.5in}{!}{\includegraphics{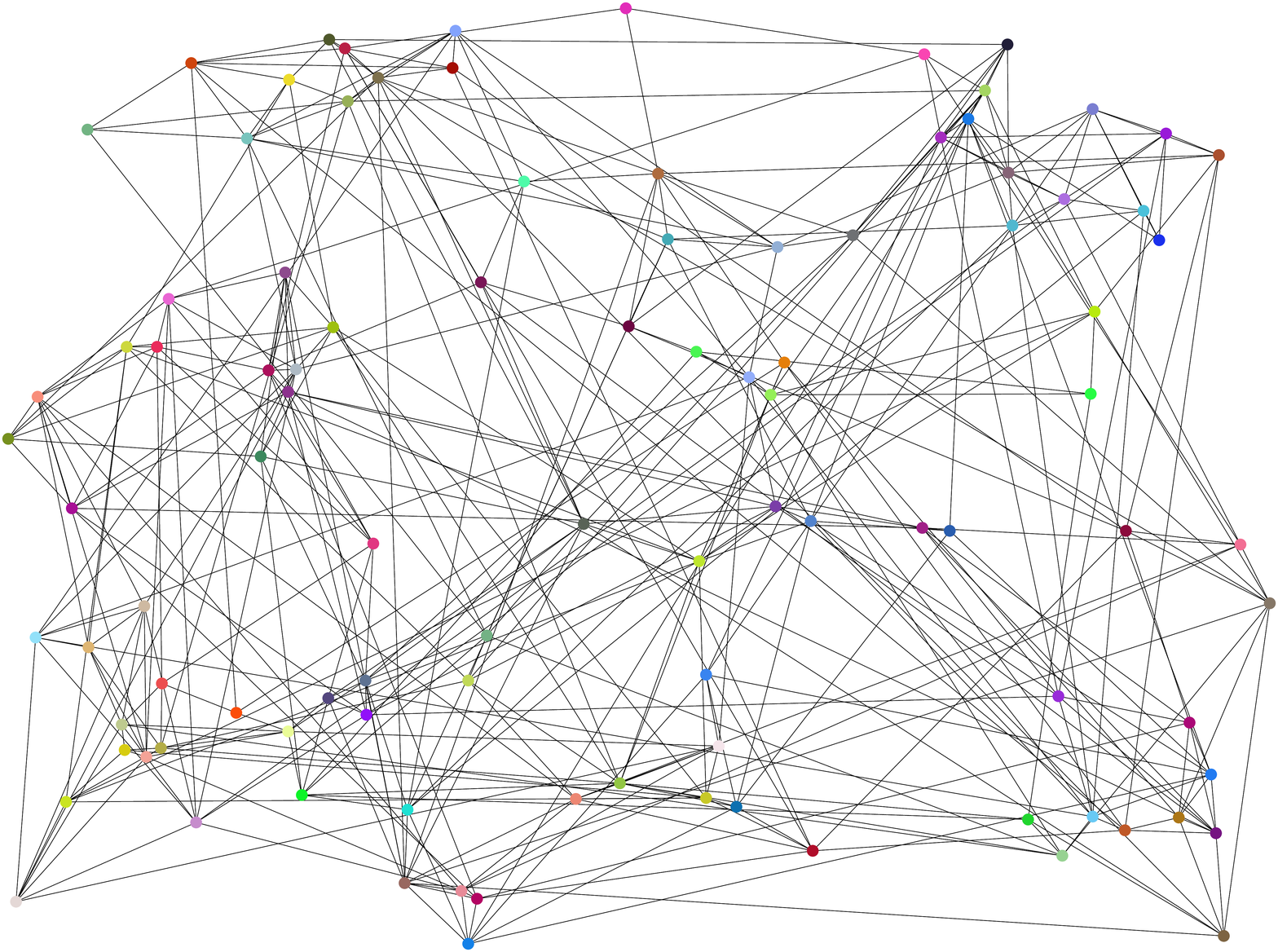}}
  \caption{A double clustering graph of 100 vertices. Each vertex has a
    random position in a two-dimensional physical space ($[0,1.33]
    \times [0,1]$), as well as a in a three-dimensional color space
    (RGB) ($[0,1]^3$), both using Euclidean distance.}
  \label{fig:dc100}
\end{figure}

Note that the two sequences are symmetric in the definition, and that
$G$ contains a nearest neighbor graph for both point sets. If, in
particular, we let the $x_i$ and $y_i$ be the vertices of two graphs
$G_1$ and $G_2$, letting $d_1$ and $d_2$ be graph distance, we may see
the construction as an augmentation of either one to create a denser
graph.

Since we are interested in probabilistic models, we want to let
$(x_i)$ and $(y_j)$ be random points. One way of doing this is to let
$\pi$ be a random permutation of $[n]$, and then letting $y_i =
x_{\pi(i)}$. In the graph case, this corresponds to:

\begin{definition}\emph{(Random Double Clustering Graph)}
  For a vertex set $V$, let $G_1 = (V,E_1)$ and $G_2 = (V, E_2)$ be a
  given graphs. Let $\pi$ be a random permutation of $V$, and
  construct $G' = (V, E')$ by letting $(u,v) \in E'$ if for all $w \in
  V$, $w \neq u, v$:
\[
d_1(u, w) < d_1(u, v) \Rightarrow d_2(\pi(u), \pi(w)) \geq d_2(\pi(u), \pi(v))
\]
where $d_1$ and $d_2$ graph distances in $G_1$ and $G_2$ respectively.
\label{def:rdc}
\end{definition}

Note that every edge added in the construction has a direction, though
in many cases (such as nearest neighbors in $G_1$ and $G_2$) edges in
both directions will be included. One may choose to see the resulting
graph as undirected by simply removing directionality and duplicated
edges. For the sake of bounding the routing time, it is advantageous
to preserve directionality and route using only outgoing edges.

In light of this, and before proceeding to analysis, we note that the
construction $G'$ works equally well if $G_1$ and $G_2$ are directed
graphs.

\section{Analysis of Double Clustering}

We will analyze special cases of Definition \ref{def:rdc}. We start by
proving that greedy routing takes only $O(\log n)$ steps in
expectation when we construct a double clustering graph using two
directed cycles. Augmenting a directed cycle is the most basic form of
Kleinberg type navigability, and has been extensively investigated in
the case of independent augmentation (see e.g.
\cite{barriere:efficient}), but of course is not a good model for most
real world scenarios.

More general models, in particular where the first space is not
directed, are more complicated. In this case the probability of
finding a link that halves the distance to the destination is not
independent of the previous route taken. This can be seen in the
simulations below, where double clustering graphs have slightly longer
greedy paths than the equivalent independent interest graphs, though
seemingly only by a constant. We attempt an analysis of one class of
such models, where the first graph may take a more general form, and
the second is an \emph{undirected} cycle (in particular, this includes
the case of two undirected cycles), but to do so we are forced to
modify the routing used somewhat. The resulting algorithm is still a
form of decentralized routing by Kleinberg's definition. Using this,
we are able to show that routing takes a polylogarithmic number of
steps, a somewhat worse bound than what we expect is true.

We conjecture that double clustering can be applied to just about any
graph (see the conclusion), but can not yet prove it.

\subsection{Two Directed Cycles}

Let $G_1$ and $G_2$ in Definition \ref{def:rdc} be two cycles of $n$
points, that is the directed graphs with vertex set $V_1 = V_2 =
\{0,1,\hdots,n-1\}$ and both $E_i$ containing an edge from $u$ to $u +
1$ (modulo $n$) for each $u \in V_i$. We will refer to this special
case as the Double Cycle Graph. It constitutes the simplest case of
double clustering.

Below, $d(x,y) = y - x \mod n$ will be graph distance in the cycles,
and $d_\pi$ will be corresponding distance function on the permuted
cycle ($d_\pi(x,y) = d(\pi(x), \pi(y))$). We will discuss greedy
routing using $d$, but by symmetry the same results hold for $d_\pi$.

Note from the definition that $G$ contains a link to the point $y$
such that $d(x,y) = 1$ (the next vertex in the cycle) and the point
$z$ such that $d_\pi(x,z) = 1$ (the next vertex in the permuted
cycle). 

Addition of vertex values below is always modulo $n$, but the notation
is suppressed for readability.

\begin{lemma}
  For $w \neq z \in V$, let $w' \in V$ be the vertex that $w$ routes to
  when $d$-greedy routing for $z$. Then:
  \begin{itemize}
  \item $w'$ lies inclusively between $w + 1$ and $z$ in the cycle
    (that is $d(w',z) < d(w,z)$).
  \item $w'$ lies inclusively between $w + 1$ and $z$ in the
    permuted cycle ($d_\pi(w',z) < d_\pi(w,z)$).
  \end{itemize}
\end{lemma}

\begin{proof}
  The first statement is obvious from the definition of greedy
  routing, and the fact that $w \leadsto w + 1$ so there is always a
  choice which approaches $z$.

  To prove the second statement, assume that $w'$ is not between $w+1$
  and $z$ in the permuted cycle. This means that $d_\pi(w, z) <
  d_\pi(w, w')$. Let $A$ be the set of points inclusively between $w'
  + 1$ and $z$ (that is $A = \{w'+1,w'+2,\hdots,z\}$).  Define $q$ as the
  first point in $A$ such that $d_\pi(w, q) < d_\pi(w, w')$, noting
  that at least one such point, $z$, exists. But by construction, and
  since $w \leadsto w'$, $q$ must be closer to $w$ in the permuted
  cycle than any vertex between $w$ and itself, and thus $w \leadsto q$.
  But if this were the case, $w$ would have routed to $q$ and not
  $w'$, which is a contradiction.
\end{proof}

\begin{corollary} 
  For any permutation $\pi$, a $d$-greedy path from any vertex $y$ to
  any other $z$ in the double cycle monotonically approaches $z$ in
  $d_\pi$, likewise a $d_\pi$-greedy path monotonically approaches $z$
  in $d$.
\label{co:mon}
\end{corollary}

In light of the corollary, it might seem that greedy routing with
respect to $d$ and $d_\pi$ would produce the same paths. In fact, this
is not the case, which we prove as an aside:

\begin{lemma}
  There exists a permutation $\pi$ such that greedy routing from some
  vertex $y$ to some vertex $z$ with respect to $d$ and $d_\pi$
  produces different paths.
\end{lemma}

\begin{proof}
  Let $\pi$, $y$ and $z$ be such that there are exactly two vertices
  $x_1$ and $x_2$ that lie between $y$ and $z$ in the cycle ($d(y,z) >
  d(x_1,z) > d(x_2,z)$), and also lie between $y$ and $z$ in the
  permuted cycle. Let $x_1$, $x_2$ appear in the opposite order the
  permuted cycle ($d(y,z) > d_\pi(x_2,z) > d_\pi(x_1,z)$).

  Note that by construction, $y$ will have edges to both $x_1$ and
  $x_2$ in the double cycle graph, because $x_1$ is closer in the
  $d_\pi$ then any $d$ closer point to $y$, and likewise for $x_2$ (in
  particular, it is closer to $y$ than $x_1$ in $d_\pi$). However,
  when greedy routing with respect to $d$ for $z$, $y$ will choose
  $x_2$, while when greedy routing with respect to $d_\pi$, it will
  choose $x_1$.
\end{proof}

Marginally, under a uniform random choice of $\pi$, the probability
that $x \leadsto y$ in the double cycle model is exactly $1 / d(x,y)$
as it should be for navigability. However, like in the all the double
clustering graphs, the random edges are not formed independently, so
the situation is different from previous results. We will see,
however, that in the case of a the double cycle, the monotonicity of
the routing path also in $d_\pi$, as proved above, makes the routing
events independent (in a sense which will be shown precisely below):
the knowledge provided by previous routing steps is always ``behind
us'' in the permuted cycle.

\begin{theorem}
  For any two points $y, z \in [n]$, the greedy path from
$y$ to $z$ in the double cycle graph formed by a uniformly random
permutation $\pi$ has expected length $O(\log n)$.
\end{theorem}

\begin{proof}
  The proof method is the same as in Theorem \ref{th:iimod}, thus we
  will consider starting in a point $x$ such that $r > d(x,z) \geq
  r/2$ and bound the expect number of steps (conditioned on the
  earlier path) until the route is within $r/2$ of $z$.

  Divide the vertices between $x$ and $z$ in the cycle into two
  equal sized sets $R$ and $H$, so that if $d(x,z)$ is odd
  \begin{eqnarray*}
    R = \{x+1,x+2, \hdots, x + {d(x,z) + 1) \over 2 }\} \\
    H = \{x+ {d(x,z) + 1 \over 2} + 1, \hdots, z - 1, z\}.
  \end{eqnarray*}
  If $d(x,z)$ is even, we let $R$ end at $x + (d(x,z) / 2)$ and $H$ go
  from there to $z-1$ so that $R$ and $H$ retain the same size.

  Note that if $x \leadsto H$, then we can route to a point with
  distance to $z$ less than $r/2$, and that
\[
\PR(x \leadsto H) = \PR(d_\pi(x,H) < d_\pi(x,R)) = 1/2
\]
where $d_\pi(x,S)$ means the minimal distance from $x$ to any point in
the set $S$.

Let $A$ be the event that $d_\pi(x,H) < d_\pi(x,R)$, and $B$ be the
event that before reaching $x$ we greedy routed along the path 
\[
x_1 \leadsto x_2 \leadsto \hdots \leadsto x_k \leadsto x
\]
for some $k$ and sequence of vertices where $d(x_i, z) < d(x, z)$. We
will show that $\PR(B \cap A) = \PR(B \cap A^c)$, which (since $P(A) =
P(A^c)$) implies that $\PR(B \given A) = \PR(B \given A^c)$ and thus
that $A$ and $B$ are independent.

To do this, we define a bijection between the set of permutations $B
\cap A$ and $B \cap A^c$. For a given $\pi \in B \cap A$, let $\pi'$
be $\pi$ composed with a permutation that flips the positions of the
elements in $R$ and $H$. Clearly, if $\pi \in A$, then $\pi' \in A^c$.

By Corollary \ref{co:mon} $d_\pi(x_i,z) > d_\pi(x,z)$ for all the
$x_i$ in the definition of $B$. This means that all the vertices in $R
\cup H$ are further from each $x_i$ than $x$ in both $d$ and $d_\pi$.
Thus the internal order of vertices in $R \cup H$ can not affect the
edges of the $x_i$, and if $\pi \in B$, then $\pi' \in B$ as
well. It follows that $|B \cap A| = |B \cap A^c|$, whence
\[
\PR(A \given B) = \PR(A) = 1/2
\]
for any $B$ defined as above. At each vertex we reach at distance
between $r$ and $r/2$ to $z$, the probability of having a link to a
vertex with distance less than $r/2$ is thus greater than $1/2$
regardless of which vertices we visited previously. The result now
follows as in Theorem \ref{th:iimod}.
\end{proof}



\subsection{Bounded Doubling Dimension and an Undirected Cycle}

In this section, we let $G_1$ belong to a more general family meeting
the criteria of Definition \ref{def:bdd}, and we let $G_2$ be an
undirected cycle (a one-dimensional toric grid). Like in previous
cases, we shall bound the expected number of steps that it takes to
halve the distance to the destination: however, unlike in previous
cases, the event of halving the distance in each step of greedy
routing is not independent of the previous path.

In order to control the dependencies between the edges encountered at
each step, we introduce a modified routing algorithm we call
\emph{half-greedy routing}. When routing for a vertex $z$ and
currently at $x$, we examine each of $x$'s neighbors in the double
clustering graph $G$. If any neighbor $w$ is such that $d_1(x, z) > 2
d_1(w, z)$, then $w$ is chosen for the next step. If no such such $w$
is found, $x$ routes to a neighbor $w'$ in $G_1$ such that $d_1(w', z)
= d_1(x,z) - 1$ (choosing from all possible such $w'$ by some
deterministic rule).

Half-greedy routing thus either takes a ``very big step'', which
immediately halves the distance, or a very small step to the next
vertex in $G_1$. Intuitively, one may imagine this as a participant in
a Milgram style experiment only bothering to send the letter by post
if he knows somebody very suitable, and otherwise just giving it
directly to one of his neighbors. The analytical advantage of this
approach is that while subsequent vertices reached by a greedy route
do not have independent positions in $G_2$, neighbors in $G_1$
(nearly) do. The navigability result thus follows from this lemma:

\begin{lemma}
  Let $\pi$ be a random permutation of $[n]$ and $d_\pi$ be circular
  distance under this permutation. That is, for $x, y \in [n]$
\[
d_\pi(x, y) = \min(|\pi(x)-\pi(y)|, n - |\pi(x) - \pi(y)|)
\]
Let $A$ and $B$ be disjoint subsets $[n]$, such that $|A| = k$ and
$|B| \geq q k$ for some $q > 0$. The elements of $A$ are enumerated $a_1,
a_2, a_3, \hdots, a_k$. Define a random variable $\tau$ by
\[
\tau = \min(t \geq 0 : d_\pi(a_t, A \backslash \{a_t\}) \geq d_\pi(a_t, B))
\]
or $\tau = k$ if this never occurs. Then, for $t < k/5$
\[
\PR(\tau \geq t) \leq e^{-m t}
\]
where $m = m(q) < \infty$, a constant independent of $n$ and $k$.
\label{lm:draw}
\end{lemma}

We will establish this lemma below. First we show how it leads to the
desired result.

\begin{theorem}
  In Definition \ref{def:rdc} let $G_1$ be a connected graph from a
  family with bounded doubling dimension, and $G_2$ be an undirected
  cycle. Then then path though $G$ between any two vertices $x$ and
  $z$ when half-greedy routing with respect to $d_1$ has expected
  length $O(\log^2 n)$.
\label{th:halfgreedy}
\end{theorem}

\begin{proof}
  Let $T$ be time it takes for half-greedy routing between any two
  vertices. We will establish the stronger fact that for $n$
  sufficiently large and a constant $h$
\begin{equation}
\PR(T \geq h (\log n)^2) \leq {\log_2 n \over n}.
\label{eq:tail}
\end{equation}
It then follows that
\begin{eqnarray*}
\E[T] & \leq & h (\log n)^2 \left (1 - {\log_2 n \over n} \right ) + n {\log_2
n \over n} \\
& = & O(\log^2 n).
\end{eqnarray*}

Fix a destination $z$, and let the phases be as in the proof of
Theorem \ref{th:iimod}. Consider the $i-th$ phase (the set of vertices
$x$ such that $2^{i-1} < d_1(x,z) \leq 2^i$), where $i$ is such that
the phase is ``big'', meaning it contains more than $\log^2 n$
vertices. We let $A$ and $B$ from Lemma \ref{lm:draw} be defined by
\[
B = B_{2^{i-2}}(z)
\]
and
\[
A = B_{5 (2^{i-1})}(z) \backslash B
\]
where the $B_r(z)$ are balls with respect to $d_1$. We note that
(distance below always means $d_1$ except where otherwise noted):
\begin{enumerate}
\item Each vertex in the $i$-th phase belongs to $A$.

\item All vertices in $B$ are within distance $3 (2^{i-1})$ of any
  vertex in the $i$-th phase.

\item Every vertex within distance $3 (2^{i-1})$ of a vertex in the
  $i$-th phase is in $A \cup B$.
\end{enumerate}
Together, these three facts mean that if a vertex $x$ in $i$-th phase
has a randomly assigned position in $G_2$ (the cycle) which is at
least as close (with respect to the permuted positions in $G_2$) to a
vertex in $B$ as any vertex other than itself in $A$, the resulting
double clustering graph $G$ will have an edge from $x$ into $B$. 

Now consider half-greedy routing starting from a vertex $x$ in the
$i$-th phase. Let the enumeration of $A$ be so that $a_1 = x$ and each
subsequent $a_j$, for $j$ up to some $\ell$, is the vertex where
$a_{j-1}$ would route if a ``very big step'' was not found. $a_\ell$
is the first vertex encountered so that it has a $G_1$ neighbor in a
lower phase, after this we may order the elements of $A$ as we wish.

Since each vertex in $B$ is less than half as far from $z$ as those in
phase $i$, the random variable $\tau$ from Lemma \ref{lm:draw} thus
dominates the time we spend in the $i$-th phase after starting from a
given vertex.

Let $b = 2/m$, where $m$ is the constant from Lemma 5.5 with $q =
1/c^4 \geq |B|/ |A|$ because of the bounded doubling dimension. Note
that $q$, and thus $m$ and $b$, are independent of which phase we are
in. Let $E^i_x$ be the event that we spend more than $b \log n$ steps
in the $i$-th phase after starting from a vertex $x$ in the phase.
Lemma \ref{lm:draw} and the argument above gives
\begin{equation*}
\PR(E^i_x) \leq {1 \over n^2}.
\end{equation*}
Since the probability is simply uniform measure of permutations of
$[n]$, this means that starting for any given vertex $x$ in the phase,
routing to the next phase will take more than $b \log n$ steps in less
than $1/n^2$ of all the permutations. Since the graph is dependent,
where we enter the phase may depend on the permutation, but the very
worst case scenario is that we always enter the phase at the vertex
where it will take the most steps to route to the next. Let $E^i$ be
the event that starting from \emph{any} vertex in the $i$-th phase, we
spend more than $b \log n$ steps in the phase. 
\begin{eqnarray*}
\PR(E_i) & = & \PR \left ( \cup_{i\text{-th phase}} E^i_x \right ) \\
& \leq & \sum_{i\text{-th phase}} \PR(E^i_x) \leq {1 \over n}
\end{eqnarray*}
where the last inequality holds because every phase trivially contains
at most $n$ vertices.

There are at most $\log_2 n$ ``big'' phases, so the probability of
spending more than $b \log n$ in any of them is less than
$\log_2 n / n$ by another union bound. Since the number of vertices in
the ``small'' phases is less than $4 \log^2 n$, (\ref{eq:tail})
follows with $h = b / \log(2)$.
\end{proof}

\vspace{0.5cm}

The remainder of this section is a proof of Lemma \ref{lm:draw}. In
order to establish the Lemma, we will make use of something we call
the \emph{toy train track} construction of a random permutation. We
equate each vertex on the cycle with a curved segments of track in toy
train set. These segments can be attached to each other to make
longer sections\footnotemark{}, and when all the $n$ segments are attached they form
a complete circle. All the pieces start in a bin, and are either red
(corresponding to vertices in $A$), blue (corresponding to vertices in
$B$), or gray (corresponding to vertices in neither set). We build the
random circular track, starting as follows:

\footnotetext{Our chosen vocabulary is to consistently use
  \emph{segment} for each element, and \emph{section} for connected
  collections of segments.}

\begin{enumerate}
\item We pick up the segment of track corresponding to $a_1$ from bin,
  this is our current section.
\item Uniformly select from the remaining pieces a segment $x$ to
  attach clockwise from the current section, and then another segment
  $y$ to attach counterclockwise from the section.
\item As long as neither the $x$ nor $y$ picked up in the last step is
  a blue or red piece, continue we draw two new pieces to attach to
  the section.
\end{enumerate}

This continues until a red or blue segment has been attached at one or
both ends of the section. At this time the first construction stage is
completed, and we \emph{put the constructed section of track back in
  the bin} together with the other pieces. If at least one end was
blue, then the building phase terminates.

If no blue piece was found, we start the second construction stage, we
try to take out $a_2$ from the bin. If $a_2$ cannot be found on its
own (it was part of the previous section), then the stage ends
immediately. If it is found, then we proceed to build a new section
starting from it as in the first stage, but this time we stop whenever
a blue segment, a red segment, or the previously constructed section
of track is attached to $a_2$'s section. At the end of stage two, we
put $a_2$'s section back in the bin as before (if one was built), and,
unless a blue piece was found, continue to stage three, which we
complete in a similar manner.

If at any time all the pieces have been added to one section the
building phase terminates, and likewise if we run out of red pieces to
start from. When the building phase has terminated, we attach
all the sections and segments in the bin in a random permutation (draw
one at a time, and attach clockwise from the last) to form a completed
circle.

Let $X$ be the number of construction stages. We make three claims
about this construction which together establish Lemma \ref{lm:draw}:

\begin{enumerate}
\item The circle of track segments created is a uniformly random
  circular permutation.
\item $X \geq \tau$ (as defined above) for the corresponding
  permutation.
\item For $t < k/5$, $\PR(X \geq t) \leq e^{-mx}$, where $m$ depends
  on $q$ but not $k$ and $n$.
\end{enumerate}

\begin{proofof}{Claim 1}
  This follows from the conditional distribution of random
  permutations. If one conditions on two segments $s_1$ and $s_2$
  being next to each other, then resulting is distribution is a random
  permutation of the remaining segments, with the $s_1$$s_2$ section
  uniformly inserted. This is equivalent to the returning of the
  section to the bin. Likewise, another section $s_3$$s_4$ would
  simply be uniformly inserted again. The claim follows from a series
  of such arguments.
\end{proofof}

\begin{proofof}{Claim 2}
  This is almost immediate. If we encounter a blue piece during
  construction stage $i$, then all the segments closer to $a_i$ than
  that piece were gray, hence $d_\pi(a_i, A\backslash \{a_i\}) \geq
  d_\pi(a_i, B)$.  If we don't find a blue piece in any construction
  stage, then $X = k$ which is an upper bound on $\tau$.
\end{proofof}

\begin{proofof}{Claim 3}
  Let $E_i$ be the event that we encounter a blue piece in the $i$-th
  construction stage. $X$ is $\min(i : E_i\text{ occurs })$ or $k$ if this is
  undefined. In the first stage, there are $k + q k$ pieces for which
  we terminate, and $q k$ are blue, so $\PR(E_1) \geq q / (1 + q) =: p$ (in
  fact greater).

  Conditioned on $E_1$ not occurring, we let $e_1$ and $e_2$ be the two
  end pieces, and note that
\[
\PR(E_2 \given E_1^c) = \PR(E_2 \given E_1^c \AND a_2 \neq e_1, e_2)
\PR(a_2 \neq e_1, e_2 \given E_1^c).
\]
If $a_2 \neq e_1$ or $e_2$, then second construction stage could
proceed. However, since $E_1$ did not occur, this means that we
removed at least two red segments from the bin, and added only one new
terminating section. Thus:
\[
\PR(E_2 \given E_1^c \AND a_2 \neq e_1, e_2) \geq \PR(E_1) \geq p.
\]
We now have to lower bound $\PR(a_2 \neq e_1, e_2 \given E_1^c)$. The
worst case is that both $e_1$ and $e_2$ are red, in which case we drew
2 red segments out of $k-1$ possible. 
\[
\PR(a_2 \neq e_1, e_2 \given E_1^c) \geq 1 - {2 \over k - 1}
\]

Using similar arguments (and rather conservative estimates), it
follows that for $i \leq k / 3$.
\begin{equation}
\PR(E_i \given E_1^c, E_2^c, \hdots, E_{i-1}^c) \geq p \left (1 -
  {2(i-1) \over k - (i - 1)} \right ) = p {k - 3(i - 1) \over k - (i -
  1)}
\label{eq:ebnd}
\end{equation}
whence
\begin{eqnarray*}
  \PR(X \geq t)  & = & \prod_{i=1}^t \PR(E_i^c \given E_1^c,
  E_2^c, \hdots, E_{i-1}^c) \\
& \leq & \prod_{i=1}^t \left (1 - p {k - 3(i - 1) \over k - (i -
  1)} \right ) \\
& \leq & \left (1
  - {p \over 2} \right )^t = e^{- m t}
\end{eqnarray*}
where $m = -\log \left(1 - {p \over 2}\right)$.
\end{proofof}

\section{Simulations}

Simulations support the conjecture that double clustering creates
navigable graphs over a larger span of structures. In cases where the
first graph is not a directed cycle, one can see that double
clustering gives slightly worse performance than when the edges are
independent, as is expected. However, the simulation data still
strongly indicates a logarithmic growth of path-length with the size of
the graph.

\subsection{Combined Greedy Routing}

Since the double clustering construction is symmetric, it should
create equally navigable networks with regard to both spaces. Thus we
can perform greedy routing in the double clustering graph with respect
to either distance function (which will sometimes lead to different
results, see below.)

A direct consequence of this is that we may try to route with respect
to to both distance functions, using at each step that which seems
most profitable. As above, assume that $z$ is the target of the route.
\begin{itemize}
\item At vertex $x$, we calculate $m_1 = d_1(w_1, z)$, where $w_1$ is
  the neighbor of $x$ which minimizes this. Similarly, calculate $m_2
  = d_2(w_2, z)$.
\item Let $n_1$ be the number of vertices within $m_1$ of $z$ in the
  first space ($M_1$) - if the space if homogeneous this is the volume
  of a ball of diameter $m_1$. Let $n_2$ be the equivalent for $m_2$
  and the second space.
\item Route to $w_2$ if $m_2$ is smaller than $m_1$, otherwise $w_1$.
\end{itemize}

We simulate combined routing as well as normal greedy routing for the
models below. In these models it seems that the benefit of using this
method regains that lost by the dependencies in the double clustering
construction: combined greedy paths are shorter than greedy paths in
the independent interest model of the same size.

\subsection{Two Undirected Cycles}

\begin{figure}
  \centering
 \resizebox{4in}{!}{\includegraphics{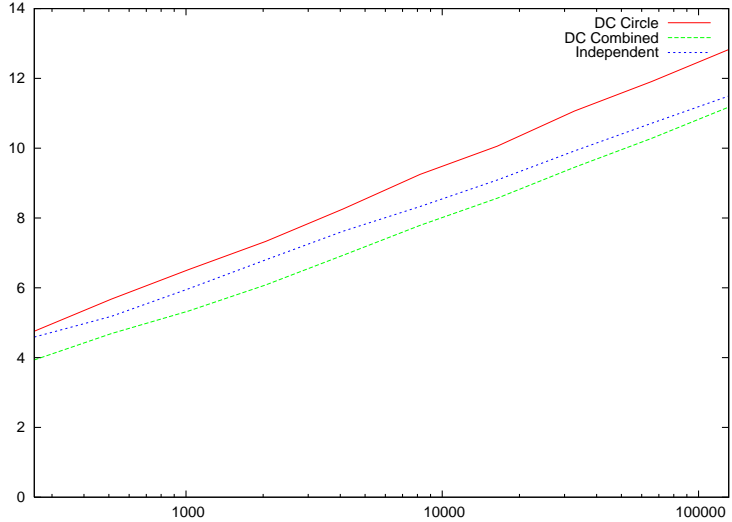}}
  \caption{Performance of double clustering graph versus the
    independent interest model using two undirected cycles.}
  \label{fig:circlecircle}
\end{figure}

The simplest undirected double clustering model is the case of
Definition \ref{def:rdc} where both $G_1$ and $G_2$ are undirected
cycles. A bound on half-greedy routing in this model is derived above,
but we can simulate also the normal greedy algorithm. The results
illustrated in Figure \ref{fig:circlecircle} - at all sizes simulated
greedy routing with respect to either cycle produces slightly longer
paths than the equivalent independent model, while combined greedy
routing produces slightly shorter paths. All lines seem to follow
strictly logarithmic growth.

\subsection{A Grid and a Tree}

\begin{figure}
  \centering
 \resizebox{4in}{!}{\includegraphics{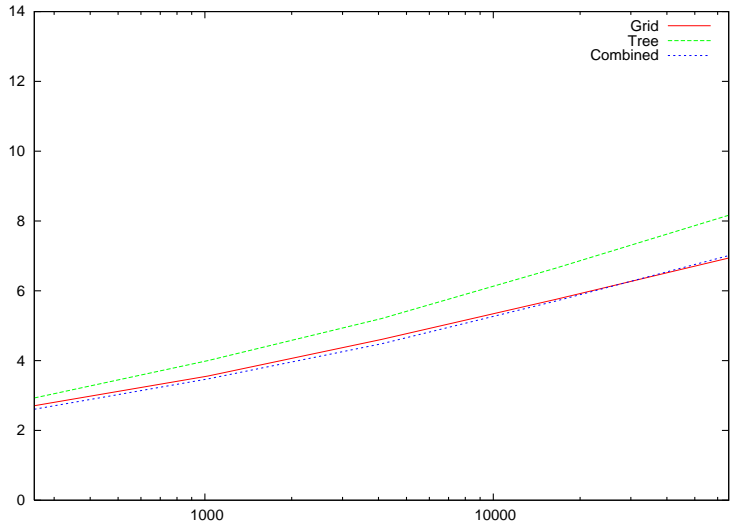}}
 \caption{Performance of double clustering using letting the first
   space (the geography) be a two dimensional grid, and the second a
   tree with the points as leaves (a categorization).}
  \label{fig:gridtree}
\end{figure}

Kleinberg's original work consisted started with a two dimensional
grid as the base graph and distance function, inspired, one expects,
by the dimensionality, if not population distribution, of the surface
of the earth. Later he \cite{kleinberg:dynamics} and Watts et al.
\cite{watts:social} proposed equivalent models based on letting
vertices have positions at the leaves of a tree. The tree represents a
hierarchical model of information, ideas, interests or other
characteristics, and the distance function is standard tree distance:
$d(x,y)$ is the depth of smallest subtree containing both. The
criteria for navigable augmentation in these cases is consistent with
(\ref{eq:genaug}).

A natural attempt at a realistic double clustering model is to combine
both of Kleinberg's models - we let the first space be a grid, and the
second be a hierarchical tree structure (in our case, a binary tree,
though any other branching is possible). We note that while the tree
distance provides a well defined metric, this space can not be seen as
a graph, so this is a sub-model of Definition \ref{def:gdc} rather than
Definition \ref{def:rdc}. A problem with the more general model is
that greedy routing is not necessarily always successful: we may reach
a vertex other than the destination with no neighbor which is closer
to the destination than itself. This can occur in this model when
routing with respect to the tree, or the combined distance, but not in
the grid (where links to neighbors in all directions always exist) -
in our simulations we simply fail and discard such routes\footnotemark{}.

Figure \ref{fig:gridtree} shows a simulation of this situation.
Routing purely using the tree shows slightly worse performance than
routing using the grid, and as such the advantage of the combined
model is less than above (for the largest data-point simulated it was,
in fact, nonexistent). As expected not all routes were successful - at
a network size of $2^{16}$ about 0.8 of the routes using only the
tree, and 0.9 of the routes using the combined routing were
successful. Another effect of the tree is that the degree of the
double clustering graph is much higher (since many vertices have the
same distance, and we only require them to be as close as any
previous.)

\footnotetext{When routing for $z$, we do allow $x$ to route to a
  vertex at the same distance as itself if no better choice, but (so
  as to not cause loops) we forbid routing to a vertex already in the
  path). This is important since tree distance has the property that
  there are a very large number of vertices at the same distance from
  any other, a large majority of the routes fail when routing for tree
  distance if this is not allowed.}

\subsection{Continuum Models}

Discrete and grid based models cannot realistically describe most
naturally occurring networks: especially social networks which are
characterized by individuals placed randomly in continuums and often
with heterogeneous population density. Continuum models for navigable
networks have been explored by Franceschetti and Meester
\cite{franceschetti:navigation} and \cite{draief:poisson} as well this
author \cite{sandberg:neighbor}, and Liben-Nowell et. al
\cite{nowell:geographic} has proposed a model based on real data that
includes non-uniform Poisson density of positions. Figure
\ref{fig:dc100} shows a simulation a continuum model with 100
vertices.


\section{Conclusion}

We have introduced a new form random graph construction, which when
combined with a random permutation of the points used to create the
graph gives rise to networks with navigable properties. These graphs
are constructed from a single natural principle, and may help explain
why networks of this type occur in real world networks.

While we have established navigability under a several cases, the
analysis presented here is far from complete. In a sense it is
unfortunate that we are able to analyze an unrealistic model (the
directed cycle) for an intuitive clear routing principle, while the
proof for the more realistic model requires somewhat contrived
routing. Theorem \ref{th:halfgreedy} also has an extra $\log n$
multiple included for technical reasons in the proof: we believe
strongly that neither this term (the actual bound is $O(\log n)$) nor
the use of half-greedy routing is actually necessary. In fact, based
on the absence of any opposing evidence in simulations or otherwise,
we believe

\begin{conjecture}
  Let $\mathcal{F}_1$ and $\mathcal{F}_2$ be two families of graphs
  with bounded doubling dimension (not necessarily with the same
  constants). For any two graphs $G_1 \in \mathcal{F}_1$ and $G_2 \in
  \mathcal{F}_2$ of size $n$, the doubling clustering graph from
  Definition \ref{def:rdc}, allows greedy routing in $O(\log n)$
  expected steps.
\end{conjecture}

Proving this in general is difficult since the structure of the two
base graphs control the dependence between the edges in the
construction. We are however hopeful that progress can be made in this
direction.  Making rigorous stronger conjectures about Definition
$\ref{def:gdc}$ is also difficult since monotonic greedy paths between
vertices may not always exist, but we believe that the resulting graph
will be navigable whenever such augmentation is possible.

Beyond this, the double clustering graph, as a new form of graph
construction, has not been analyzed for questions other than
navigability. Questions such as connectivity, diameter, and edge
length remain open in some or all cases. And, finally, the question of
how well double clustering actually matches the real world has not
been investigated.
